\documentclass[letter,11pt]{article}

\usepackage{pifont}

\usepackage{comment}
\usepackage[linktocpage=true]{hyperref}
\usepackage[margin=1.2 in, top=1 in, bottom= 1.2 in]{geometry}

\usepackage[english]{babel}
\usepackage{amsfonts}
\usepackage{mathrsfs}
\usepackage[mathscr]{euscript}
\usepackage{bbm}
\usepackage{latexsym}
\usepackage{math dots}
\usepackage{amssymb}
\usepackage{mathtools}

\usepackage{enumitem}
\setlist{nolistsep}
\usepackage{amsthm}
\usepackage{relsize}
\usepackage{nicefrac}
\usepackage[capitalize]{cleveref}

\newtheoremstyle{plain}{3mm}{3mm}{\slshape}{}{\bfseries}{.}{.5em}{}
\newtheoremstyle{definition}{2mm}{2mm}{}{}{\bfseries}{.}{.5em}{}
\theoremstyle{plain}
	
\newtheorem{theorem}{Theorem}

\newtheorem{conjecture}{Conjecture}

\theoremstyle{definition}

\theoremstyle{plain}
\newtheorem*{namedthm}{\namedthmname}
\newcounter{namedthm}
	

\newcommand{\N}{\mathbb{N}}
\newcommand{\Z}{\mathbb{Z}}

\newcommand{\eps}{\epsilon}

\title{Difference sets in quadratic density Hales Jewett conjecture with 2 letters}
\author{Aritro Pathak}

\begin{document}

\maketitle

\begin{abstract}
The quadratic density Hales--Jewett conjecture predicts that any sufficiently dense subset of $\{0,1\}^{n^{2}}$ contains a combinatorial line whose wildcard set has a prescribed quadratic structure. In this paper, we prove a weaker but structurally natural statement: for all sufficiently large $n$, every dense subset of $\{0,1\}^{n^{2}}$ contains two elements whose difference set is a Cartesian product $\gamma_{1}\times \gamma_{2}$, where $\gamma_{1},\gamma_{2}\subset\{1,\dots,n\}$ are nonempty. Our argument is elementary and quantitative. We also show that imposing additional structural restrictions on the wildcard set allows the construction of dense subsets avoiding such combinatorial lines, demonstrating sharp obstructions to certain strengthened forms of the conjecture. These results connect to recent work on structures of symmetric differences of dense sets of graphs on a given vertex set.
\end{abstract}

\section{Introduction}

Consider for any $n\in \Z_+$, the $n \times n$ grid where each entry $\{(i,j):1\leq i,j\leq n \}$ is filled with either 0 or 1, i.e the set $\{0,1\}^{n^{2}}$.

The following is the first case of a central open conjecture in Ramsey theory, as mentioned by Gowers\cite{Gowersblog}, and Bergelson. This question has been brought up again in a blog post by Kalai\cite{Kalaiblog}.

\begin{conjecture}[Quadratic Density Hales Jewett with 2 letters]\label{conj1}
For any $0<\delta<1$, there exists $QDHJ(\delta)$ so that for any $n\geq QDHJ(\delta)$, for any subset $S\subset \{0,1\}^{n^2}$ with $|S|\geq \delta \cdot 2^{n^2}$, there exist two elements $s^{(0)}, s^{(1)}\in S$ such that the set $\{(i,j):s^{(0)}_{ij}\neq s_{ij}^{(1)}\}$ is the same where $\{ s^{(0)}_{ij}=0,s^{(1)}_{ij}=1\}$, where $\{(i,j)\in \gamma\times \gamma, \ \gamma\subset \{1,2,..,n\}\}$ with $\gamma$ a nonempty subset of $\{0,1\}^{n^2}$. 
\end{conjecture}

This is the quadratic base case with $k=2$ letters of the general Polynomial Density Hales Jewett conjecture stated later as \cref{PDHJ}. To state this conjecture formally, we first introduce some notation, which we essentially borrow from Walters \cite{Walters}.

For any given $n\in \mathbb{Z}_{+}$, consider the set of words of length $n$, with each letter of the word being an element of $[k]:=\{1,2,..,k\}$. Formally, this is written as $K=[k]^{n}$. For any $a\in K, \gamma \subset [n]$ and $1\leq x\leq k$ we denote by $b=a\oplus x\gamma$ the element of $[k]^{n}$ which is obtained by setting $b_i=x$ if $i \in \gamma$ and $b_i=a_i$ otherwise. A \textit{combinatorial line} is a set of the form $\{a\oplus x\gamma :1\leq x\leq k\}$. We call $\gamma$ the wildcard set for the combinatorial line. A \textit{k-dimensional combinatorial subspace} is a set of the form $\{a\oplus x_1 \gamma_1 \oplus x_2 \gamma_2 \dots \oplus x_k \gamma_k : 1\leq x_{i}\leq k, \forall 1\leq i\leq k \}$

In the context of the Polynomial Hales Jewett theorem, we are looking at not just a linear coordinate space such as $\{1,2,\dots,n\}$ but at $d$-dimensional coordinate grids, $\{1,2,\dots,n\}^{d}$ for all positive integers $d$. In these cases, for the set of words of length $n^d$ with $k$ letters, we use the symbol $[k]^{n^{d}}$. When $k=2$, for any two words $s_1,s_2 \in \{0,1\}^{n^{d}}$, we call the \textit{difference set} of $s_1,s_2$ the set of coordinate points where the words differ. 

We state the Polynomial Hales Jewett theorem in a form articulated by Walters \cite{Walters}, and later by McCutcheon\cite{McCutcheon}, and which is a generalization of the original theorem proven by Bergelson and Leibman\cite{Bergelson1}.

\begin{theorem}[Polynomial Hales Jewett:]\label{thm:polyhales}
 For positive integers $k,r$, there exists a positive integer $N(k,r)$ such that for all $n\geq N(k,r)$ whenever $K[n]:=[k]^{n} \times [k]^{n^{2}}\times \dots \times [k]^{n^{d}}$ is $r$ colored, there exists $a\in K[n]$ and $\gamma\subset [N]$ such that the set of points $\{a\oplus x_1\gamma \oplus x_2 (\gamma \times \gamma)\oplus \dots \oplus x_{d} \gamma^{d}:1\leq x_{i} \leq k\}$ is monochromatic.
\end{theorem}

As noted by Walters\cite{Walters}, the original Polynomial Van der Waerden theorem follows straightforwardly from this Polynomial Hales Jewett theorem. 

Stronger density versions of these coloring statements have been established over the years. The density version of van der Waerden's theorem\cite{Waerden} is Szemeredi's theorem, of which several celebrated proofs are known. \cite{Szemeredi,Furstenberg2,Gowers2, Gowers3, Gowers4, Nagle, Rodl}.

\begin{theorem}[Szemeredi:]\label{thm:szemeredi}
For every positive integer $k$ and every $\delta>0$ there exists a positive integer $N(k,\delta)$ so that for any $n\geq N(k,\delta)$ and every subset $A\subset [n]$ of size at least $\delta\cdot n$ contains an arithmetic progression of length $k$.
\end{theorem}

It is easy to show that the original Hales Jewett theorem\cite{Hales} implies the classical van der Waerden theorem by considering the base $k$ representations of the integers in question. In 1991, Furstenberg and Katznelson proved the density version of Hales Jewett theorem \cite{Furstenberg1} from which Szemeredi's theorem follows as a corollary. Later several other proofs were also found\cite{Polymath, Austin, Dodos}.

We state the Density Hales Jewett theorem here:

\begin{theorem}[Density Hales Jewett:]\label{thm:densityhales}
For every positive integer $k$ and any $0<\delta<1$, there exists a positive integer $DHJ(k,\delta)$ such that for any $n\geq DHJ(k,\delta)$ , a subset $A\subset[k]^{n}$ with density at least $\delta$ contains a combinatorial line of size k.
\end{theorem}

The $k=2$ version of the density Hales Jewett theorem follows from Sperner's theorem; where we interpret a string of $0'$s and $1'$s of length $n$ to be a subset of $\{1,2,...,n\}$, and a combinatorial line is interpreted as two subsets $A,B$ such that one of them, say $A$ is contained inside $B$. Thus to not contain a combinatorial line, we are looking at a maximal antichain which by Sperner's theorem has size at most ${{n}\choose {\lfloor \frac{n}{2} \rfloor}}$, and by Stirling's approximation, $\frac{{{n}\choose {\lfloor \frac{n}{2} \rfloor}}}{2^{n}}\sim \frac{2}{\sqrt{n}} \to 0$ as $n\to \infty$. Thus for large enough $n$, we are forced to contain a combinatorial line.

Later, Bergelson established the polynomial van der Waerden theorem and polynomial Szemeredi theorem in the same paper\cite{Bergelson2}, for polynomials with integer coefficients and without constant terms. Such polynomials are refered to as integral polynomials. We state the theorems below:

\begin{theorem}[Polynomial van der Waerden:]\label{thm:polyvander}
Suppose that $p_1,p_2,..,p_m$ are $m$ integral polynomials. Then there exists a positive integer $N(m)$ such that for all $n\geq N(m)$, there exists $a$ and $0<d\leq n$ such that the set $\{a\} \cup \{a+p_i(d):1\leq i\leq m\}$ is monochromatic, whenever $\N$ is finite colored.

\end{theorem}

\begin{theorem}[Polynomial Szemeredi:]\label{thm:polyszemeredi}

Suppose that $p_1,p_2,..,p_m$ are $m$ integral polynomials and $0<\delta<1$ any small real number. Then there exists a positive integer $N(m, \delta)$ such that for all $n\geq N(m,\delta)$, and any set $A\subset [n]$ of size at least $\delta\cdot n$, there exists $a$ and $0<d\leq n$ such that the set $\{a\} \cup \{a+p_i(d):1\leq i\leq m\}$ is monochromatic.

\end{theorem}

There is a natural polynomial density Hales Jewett conjecture that generalizes \cref{thm:polyszemeredi}, which we state below. It would be a generalization of both the Polynomial density Hales Jewett theorem as well as the Density Hales Jewett theorem. For $k=2$, and only considering the quadratic term, we stated the corresponding version in \cref{conj1}.

\begin{conjecture}[Polynomial Density Hales Jewett]\label{PDHJ}

For positive integers $k$ and a small $0<\delta<1$, there exists a positive integer $N(k,\delta)$ such that for all $n\geq N(k,\delta)$ whenever we consider $A\subset K[n]=[k]^{n} \times [k]^{n^{2}}\times \dots \times [k]^{n^{d}}$ with cardinality at least $\delta\cdot k^{n+n^{2}+\dots+n^{d}}$, there exists $a\in K[n]$ and $\gamma\subset [N]$ such that the set of points $\{a\oplus x_1\gamma \oplus x_2 (\gamma \times \gamma)\oplus \dots \oplus x_{d} \gamma^{d}:1\leq x_{i} \leq k, 1\leq i\leq d \}$ is contained in $A$.
\end{conjecture} 




In this paper, we demonstrate there are dense subsets of $\{0,1\}^{n^{2}}$ that have certain restrictions on the shape of the wildcard sets for the combinatorial line. These restrictions allow for the construction of dense subsets avoiding such combinatorial lines. 

Here we state the first main result of this paper.

\begin{theorem}\label{thm:mainthm1}

\begin{enumerate}

    \item[(i)] Given a positive integer $k$ and large enough $n$, and a set of elements $S=\{a_1,\dots,a_k\} \subset \{1,\dots,n\}$, there exist subspaces of $\{0,1\}^{n^2}$ of density $\delta=\frac{1}{2^{k}}$ so that the difference set of any two elements of this subspace that constitute a combinatorial line, does not contain any coordinate belonging to $S$.


    \item[(ii)] For any large enough $n\geq 4$, there exist subspaces of density $\delta=\frac{1}{4}$ so that the difference set corresponding to the combinatorial line contained in this subspace is of the form $\gamma \times \gamma$ with $\gamma\subset \{1,\dots,n\}$ so that we can only have $4 ||\gamma|$.
    
    \item[(iii)] For any large enough $n\geq 4$, there exist subspaces of density $\delta=\frac{1}{2}$ so that the difference set corresponding to the combinatorial line contained in this subspace is of the form $\gamma \times \gamma$ with $\gamma\subset \{1,\dots,n\}$ so that we can only have $|\gamma|\equiv 0,-1 (\text{mod} \ 4)$.
    
    
    
\end{enumerate}
\end{theorem}    

In particular, consider the statement in \cref{thm:mainthm1}, which shows one can construct dense subspaces so that the difference sets can avoid cliques with lengths that have residues $0,-1 \ \text{mod}(4)$.\footnote{The reader can compare this, with the results in \cite{Alo24}, Theorem 1.2, Propositions 1.3-1.5.}

\bigskip
We then prove the other main result of this paper:

\begin{theorem}\label{thm:mainthm2}
For any $\delta>0$, there exists $N(\delta)= \lceil \log_2(\frac{1}{\delta}+1)  \rceil$ such that for all $n\geq N(\delta)$, for any nonempty $\gamma_1\subset \{1,2,\dots,n\}$, we have for any subset $S$ of $\{0,1\}^{n^{2}}$ with $|S|\geq \delta\cdot 2^{n^{2}}$ , two elements $s_{1},s_{2}\in S$ such that their difference set is of the form $\gamma_{1} \times \gamma_{2}$ where $ \gamma_2\subset\{1,2,\dots,n\}$ is a non empty subset.   
\end{theorem}

One also refers to the \cite{Alo24} to consider specific other shapes that are obtained, when taking symmetric differences of elements of a dense subset of $\{0,1\}^{n^2}$. See also, the recent advances in \cite{EFF12, BZ23} on the maximal sizes of triangle intersecting where also one works with certain affine subspaces of $\Z^{n}_2$, and $K_4$ intersecting families of graphs.


Following Furstenberg, Katznelson's original argument that showed that the multidimensional density Hales Jewett theorem follows from the density Hales Jewett theorem (Proposition 1.7 in \cite{Polymath}), we show that the natural conjectured multidimensional analog of the polynomial density Hales Jewett theorem follows from the conjectured quadratic density Hales Jewett theorem. 

\begin{theorem}\label{thm:mdqhj}($MQDHJ_{k,d}$) Suppose the polynomial density Hales Jewett theorem is true. Let $\epsilon >0, k\in \mathbb{N}$. There exists $N=N(\epsilon,k)$ having the property that if $E\subset \{1,..,k\}^{N^{2}}$ with $|E|\geq \epsilon k^{N^{2}}$, then for each $d\geq 1$ there exists a $d$ dimensional combinatorial subspace contained in $E$, of cardinality $k^{d}$, where each of the $d$ wildcard sets is of the form $\alpha_i \times \alpha_i$ where for $1\leq i\leq d$ , each of the $\alpha_i \subset \{ 1,2,..,N\}$ are distinct. 
\end{theorem}

\begin{proof}[Proof of \cref{thm:mdqhj}]
The induction is on the dimension $d$. Suppose we have proved \\ $MQDHJ_{k,d-1}$ for $d-1$ and let $E\subset [k]^{N^{2}} $ have density $\epsilon>0$. Let $m=MQDHJ_{k,d-1}(\epsilon/2)$. For every string $z\in [k]^{N^{2}}$, consider the associated pair $z'=(z_{u},z_{l}, z_s)$ where $z_{u}$ is the restriction of $z$ to the upper left $m^{2}$ square coordinates, and $z_{l}$ is the restriction to the lower right $(N-m)^{2}$ square coordinates, and $z_s$ is the restriction of $z$ to the $(N-m)\times m$ rectangle on the bottom left and the $m\times (N-m)$ rectangle on the top right.

For any $w\in [k]^{2m(N-m)} $ where the coordinates are the union of the bottom left $(N-m)\times m$ and the top right $m\times (N-m)$ coordinates,  define the set $E_w := \{(z_u,z_l)\in E: z_u \in [k]^{m^2}, z_l \in [k]^{(N-m)^2}, (z_u,z_l,w)\in E\}$ . It follows that there exists some $w\in [k]^{2m(N-m)}$, so that $E_w$ has density at least $\eps$ within $[k]^{m^2 +(N-m)^2}$, which is the the union of the top left $m\times m$ and the bottom right $(N-m)\times (N-m)$ coordinates. If such a $w$ does not exist, then the density of $E$ is crudely bounded from above by $k^{2m(N-m)}\eps k^{m^2 +(N-m)^2)}=\eps k^{N^2}$ which is a contradiction. Henceforth we choose a particular such $w_0$ and work with $E_{w_0}$.


 For any $(z_u,z_l,w_0)\in E_{w_0}$,  call $z_l\in[k]^{(N-m)^{2}}$ good, if the set $E_{z_{l}}= \{ z_{u} \in [k]^{m^2}: (z_{u},z_{l}, w_0)\in E \}$ has density at least $\epsilon/2$ within $[k]^{m^{2}}$. It is not hard to see that the set of good strings has to have density at least $\epsilon/2$ within $[k]^{(N-m)^{2}}$. If not, we can bound from above the total number of elements in $E$ by $\Big((\epsilon/2)\cdot k^{m^{2}}\cdot k^{(N-m)^{2}} + (\epsilon/2)\cdot k^{(N-m)^{2}}\cdot k^{m^{2}} \Big)=\epsilon \cdot k^{m^{2}+(N-m)^{2}}$ which is a contradiction.

By induction, for any good $z_{l}$ the set $E_{z_{l}}$ contains a $(d-1)$ dimensional subspace. There are only a finite number of possible $(d-1)$ dimensional subspaces in $[k]^{m^{2}}$, and we call that number $T(m)$. By our choice of $m=MQDHJ_{k,d-1}(\epsilon/2)$, for any given good $z_{l}$, the set $E_{z_{l}}$ must have a $(d-1)$ dimensional subspace. 

For each such $(d-1)$ dimensional subspace $\sigma$ in the upper left set of squares of cardinality $k^{m^{2}}$ construct the set $G_{\sigma}=\{ z_{l}:(x,z_{l}, w_0)\in E  \ \  \forall x\in \sigma \}$ which is a subset of the lower right square of dimension $(N-m)\times (N-m)$. The union of the sets $G_{\sigma}$ over all possible $(d-1)$-dimensional subspaces covers the set of good $z_{l}$'s. Since the set of good strings has a density at least $\epsilon/2$ within $[k]^{(N-m)^{2}}$, at least one of the sets $G_{\sigma_{0}}$ has density at least $\frac{\epsilon}{2T(m)}$ within $[k]^{(N-m)^{2}}$. Assuming the quadratic density Hales Jewett statement to be true, we can choose $N-m> QDHJ_{k}(\frac{\epsilon}{2T(m)})$ and then we would get a 1-dimensional subspace $\lambda$ in the lower right set of squares, and then the subspace $\sigma_{0}\times \lambda$, each element of which has a fixed $w_0$ when restricted to the bottom left and top right coordinates, suffices for our purposes.
\end{proof}

\section{Symmetric differences and addition in the vector space $\mathbb{F}_{2}^{n^{2}}$}

Taking symmetric differences in our case is equivalent to the properties of addition in the vector space $\mathbb{F}_{2}^{n^{2}}$ over the field $\mathbb{F}_{2}$. Henceforth, we talk about addition in $\mathbb{F}_{2}^{n}$ in place of symmetric differences. We can get exactly $M=2^{m}-1$ many non-zero elements in a vector subspace upon addition in $\mathbb{F}_{2}^{n^{2}}$, when considering a vector subspace of $\mathbb{F}_{2}^{n^{2}}$ with $m$ basis vectors.

Let $m\leq n^{2}$. We consider any arbitrary $m$ elements $a_1,a_2,\dots, a_m$ among which there are no nontrivial relations through symmetric differences, this being possible since $\mathbb{F}_{2}^{n^2}$ is $n^{2}$ dimensional and $m\leq n^{2}$.

Considering all the sums from the basis set $\{a_1,\dots,a_m\}$, with coefficients in $\mathbb{F}_2$, we get exactly $M=(2^{m}-1)$ non zero elements. We enumerate these elements as $a_1,a_2,\dots,a_M$ keeping the first $m$ terms as the basis elements. This is a finite dimensional vector subspace with the sumset taking as many values as the number of elements in the subspace: for each $a_{k}, 1\leq k \leq M$, we have a unique chain of $2^{m-1}-1$ pairs of elements that add up to $a_k$; for $a_{k}$ and any other $a_l, l\neq k, 1\leq l\leq M$, the element $(a_{k}-a_{l})\in \mathbb{F}_{2}^{n^{2}}$ is added to $a_{l}$ to get $a_{k}$.\footnote{The question of finding dense subsets of $\mathbb{F}_{2}^{n}$ and checking whether the sumset has a large vector subspace, has been studied and there is an example of a set of density $1/4$ whose sumset does not contain a subspace of dimension $\sqrt{n}$. \cite{Green} } 

\section{Proof of \cref{thm:mainthm1}}
Within our $[n]^{2}$ grid, the first coordinate increases downwards, starting from 1, along a column when the second coordinate is fixed, and along a row the second coordinate increases, starting from 1, for a fixed value of the first coordinate. Given any $\delta=\frac{1}{2^{k}}$, in each case below we construct  specific subsets of $\{0,1\}^{n^2}$ of cardinality $\delta \cdot 2^{n^{2}}=2^{n^{2}-k}$. 

By a \textit{coordinate point} in $\{0,1\}^{n^2}$, we mean an element of the form $(x,y)\in \{0,1\}^{n^2}$ with $0\leq x,y\leq n$. Alternately, we also call such an element $(x,y)$ as a single element subset of $[n]^2$.

We recall that given two elements $a_1,a_2 \in \mathbb{F}^{n^2}_2$, the coordinate points in $\{0,1\}^{n^2}$ where they differ are precisely those coordinate points where the element  $(a_1 +a_2 )\subset \mathbb{F}^{n^2}_2$ takes the value $1$. 

We can construct $(n^{2}-k)$ many linearly independent subsets of $\mathbb{F}_{2}^{n^2}$ over $\mathbb{F}_{2}$, and get $2^{n^2 -k}-1$ many non-zero elements.

\begin{proof}[Proof of Theorem 6]

(i): In this case, consider a vector subspace consisting of all the coordinate points in $\{0,1\}^{n^2}$, except for the diagonal points $\{a_1,a_1\},\{a_2,a_2\},\dots,\{a_k,a_k\}$. In this case, it is clear that one cannot construct a combinatorial line which contains any one of the points from the set $S$.


(ii) Consider the set of all elements of $\{0,1\}^{n^2}$ that have an even number of elements in the diagonal; which is clearly a set of density $\frac{1}{2}$. Within this set, consider further the subset that has an even number of elements in the upper triangular part of the $[n]^{2}$ grid. Thus this set has a density of $\frac{1}{4}$ in $\{0,1\}^{n^2}$. In order to get a wildcard set, we would only get an even number of elements on the diagonal. Further we could only have an even number of elements in the upper triangular part of the grid. If we considered $4k+2$ many elements on the diagonal, then we would have to consider $1+2+\dots+(4k+1)=(2k+1)(4k+1)$ many points in the upper triangular grid, and this is a contradiction. Thus the only possibility is to get wildcards sets $\gamma \times \gamma$ with $|\gamma|$ divisible by 4. 

(iii) Consider the basis set in this case, which has the following three types of elements: firstly the elements $\{(x,y)\}, \forall x,y\in [n]$ where $ x>y$, then the elements of the form $\{(t,t),(t+1,t+1)\}$ with $t\in [n-1]$, and further the elements of the form $\{(x,y),(x,y+1)\}$ when $x\leq y$ and when $(y+1)\leq n$. Consider any wildcard set of form $\gamma\times \gamma$ with $|\gamma|=k$ that is obtained, and enumerate the elements in $\gamma$ in increasing order as $\{t_1,\dots,t_k\}$. In order to include the point $(t_k,t_k)$, we have to include all the 2-element basis sets on the diagonal; $\{(t_{k}-1,t_{k}-1),(t_k,t_k)\}+\{(t_{k}-2,t_{k}-2),(t_{k}-1,t_{k}-1)\}+\dots+\{(t_{k-1},t_{k-1}),(t_{k-1}+1,t_{k-1}+1)\}=\{(t_{k-1},t_{k-1}),(t_k,t_k)\}$. Further, to include the point $(t_{k-1},t_{k})$ the only option we have is to include the basis sets $\{(t_{k-1},t_{k-1}),(t_{k-1},t_{k-1}+1)\},\dots, \{(t_{k-1},t_{k}-1),(t_{k-1},t_{k})\}$ which add up to the set $\{(t_{k-1},t_{k-1}),(t_{k-1},t_{k})\}$.

 Till now the point $(t_{k-1},t_{k-1})$ was added exactly twice and is thus zero. We now have to further include all the two element diagonal elements starting from $(t_{k-2},t_{k-2})$ till $(t_{k-1},t_{k-1})$ (in a  way similar to how we included all the elements from $(t_{k-1},t_{k-1})$ till $(t_{k},t_{k})$ in the previous paragraph). We would now have to include the elements $\{(t_{k-2},t_{k-1}),(t_{k-2},t_k)\}$ by adding up a suitable number of basis elements in the upper triangular region. This would be permissible as a wildcard set of shape $\gamma\times \gamma$ with $|\gamma|=3$. We can extend this to have $|\gamma|=4$ where we have four further points to account for in the same row with the first coordinate $t_{k-3}$, which can be done as before, and this is the only way to account for the point $(t_{k-3},t_{k-3})$ (If we accounted for the point $(t_{k-3},t_{k-3})$ by covering it with the set $\{(t_{k-3}-1),(t_{k-3}-1),(t_{k-3},t_{k-3})\}$ along the diagonal, then we would be left to cover three elements in the upper triangular part in the row $t_{k-3}$ which is impossible with our chosen basis elements). So in this way, we can get a wildcard set of size either $3$ or $4$. It is now clear that we can extend the argument to get wildcard sets only of size $4k$ or $4k-1$ for any $k\geq 2$. We have one element basis sets in the lower triangular part of the grid, and so we can add each required coordinate point there. In this process we have exactly $2^{n^{2}-1}$ many elements in this subspace we considered, and thus it is a subspace of density $\frac{1}{2}$.
 
 

\end{proof}

\section{Rectangular and square wildcard sets}

\begin{proof}[Proof of \cref{thm:mainthm2}]
Observe that $(\gamma_a \times \gamma_b)+(\gamma_a \times \gamma_c)=\gamma_a \times (\gamma_b +\gamma_c)$ where addition is in $\mathbb{F}_{2}^{n^2}$.

Consider for any fixed $\gamma$, all the nonempty sets of the form $\gamma_{i}:=\gamma\times t_i$ with $t_i\subset \gamma$. There are $2^{n}-1$ many such sets. Suppose we have a subset $S\subset \{0,1\}^{n^2}$ with density $\delta$, whose elements we enumerate as $s_1.s_2,\dots,s_{|S|}$. If there exist some $\gamma_i, \gamma_j$ such that $\gamma_i+s_{t}=b, \gamma_j +s_{l}=b$ for the same fixed $b\in \{0,1\}^{n^2}$, then $\gamma_i+\gamma_j=\gamma\times (t_i+t_j)=s_{l}+s_{t}$ (Clearly $\gamma\times (t_i+t_j)$ is of the form $\gamma\times (t_k)$ for some nonzero $t_k\subset \gamma$). Such a thing is forced for values of $n \geq \lceil \log_{2}\big( \frac{1}{\delta} +1  \big)\rceil$: consider the set of all pairs of elements $(\gamma_i,s_{t})$ of which there are $(2^{n}-1)\cdot \delta\cdot 2^{n^2}$ many. If for any pair $(\gamma_{i},s_{t})$ we have $\gamma_i +s_t =s_j$ for some $j\in [S]$, we are done since then $s_i +s_j= \gamma_i$. Otherwise each of these pairs sum to some element in $\{0,1\}^{n^{2}}\setminus S $, and there are $(1-\delta)\cdot 2^{n^2}$ many such elements. Thus there has to be some fixed $b$ to which two pairs sum to, in which case again we are done as before.
\end{proof}

We could specialise the above argument where the rectangular difference set obtained actually has a form $\gamma_1 \times \gamma_2$ with $|\gamma_1|=|\gamma_2|$, but where the sets $\gamma_1$ and $\gamma_2$ are not necessarily the same. That is the content of the first part of the following theorem\footnote{After the completion of this work, it came to the notice of the author that this theorem has also been mentioned in a different language, in the end of the blog post of \cite{Gowersblog}.}. Also, if instead of square wildcard sets, we require wildcard sets that are the disjoint unions of two square wildcard sets, i.e. of the form $\gamma_1\times\gamma_1 +\gamma_2\times\gamma_2$ with $\gamma_1,\gamma_2\subset [n]$, then again the proof follows with essentially the same basic argument as above.

\begin{theorem}\label{thm:mainthm3}
(i)For any $\delta>0$ and even integer $l=2l_1>0$, there exists $N(\delta)= \lceil \frac{l}{2\delta}  \rceil$ such that for all $n\geq N(\delta)$, for any nonempty $\gamma_1\subset [n]$ with $|\gamma_1|=l$, we have for any subset $S$ of $\{0,1\}^{n^{2}}$ with $|S|\geq \delta\cdot 2^{n^{2}}$, two elements $s_{1},s_{2}\in S$ such that their difference set is of the form $\gamma_{1} \times \gamma_{2}$ with $|\gamma_2|=l$ as well.

(ii) For any $\delta>0$ and integer $l>0$, there exists $N(\delta)= \lceil \frac{l}{\delta}  \rceil$ such that for all $n\geq N(\delta)$, we have for any subset $S$ of $\{0,1\}^{n^{2}}$ with $|S|\geq \delta\cdot 2^{n^{2}}$, two elements $s_{1},s_{2}\in S$ such that their difference set is of the form $\gamma_1\times\gamma_1 +\gamma_2\times\gamma_2$ with $\gamma_1,\gamma_2\subset [n]$ with $|\gamma_1|=|\gamma_2|=l$.
\end{theorem}

\begin{proof}[Proof of \cref{thm:mainthm3}]

(i) Consider all pairs of elements of the form $(\gamma_1 \times \Tilde{\gamma},s)$ where $\Tilde{\gamma}$ are all the $l_1$-element subsets of $[n]$ of the form $\{tl_1+1,\dots,(t+1)l_1\}$ for $0\leq t<\frac{n}{l_1}$, and $s\in S$, where $n\geq N(\delta)$ and $S$ is as in the statement of the theorem. Clearly we have $(\frac{n}{l_1})\cdot\delta\cdot 2^{n^{2}}$ many distinct such pairs. In this case, when $\frac{n\delta}{l_1}\geq 1$, i.e. when $n\geq N(\delta)$, we must have that two distinct pair sums are the same, and then it is clear that we will find $s_1+s_2=\gamma_1\times \gamma_2$ with $\gamma_2=\Tilde{\gamma_1}+\Tilde{\gamma_2}$ with both $\Tilde{\gamma_1},\Tilde{\gamma_2}$ consecutive $l_1$ element subsets.

(ii) In this case, we consider pairs $(\gamma\times \gamma,s)$ with $\gamma\subset [n]$ of the form $\{tl_1+1,\dots,(t+1)l_1\}$ for $0\leq t<\frac{n}{l}$, and $s\in S$ and the argument proceeds as above.
\end{proof}

\section{Further questions.}

We state three natural conjectures below of which the first two are weaker versions of the density quadratic Hales Jewett conjecture.

\textit{Conjecture 1:} For any $0<\delta<1$, there exists $N(\delta)$ such that for any $n\geq N(\delta)$, for any subset $S\subset \{0,1\}^{n^{2}}$ we must have two elements $s_1,s_2\in S$ such that their difference set is of the form $\gamma\times \gamma$ with $\gamma \subset [n]$.  

\bigskip

\textit{Conjecture 2:} For any $0<\delta<1$, there exists $N(\delta)$ such that for any $n\geq N(\delta)$, for any subset $S\subset \{0,1\}^{n^{2}}$ we must have two elements $s_1,s_2\in S$ such that their difference set is of the form $\gamma_1\times \gamma_2$ with $\gamma_1,\gamma_2 \subset [n]$ and further $\gamma_1\times \gamma_2$ is the wildcard set of a corresponding combinatorial line .  

\bigskip

We define a restriction on $[n]$, as the property that in any nonempty subset of $[n]$, a given pattern or sequence of elements is forbidden from appearing. As an example, say that a subset $S\subset [n]$ satisfies a specific restriction of there being no single element in isolation in the subset $S\subset [n]$, which means there cannot exist any $a\in S$ such that $a-1\notin S,a+1\notin S$. Another example is where we have a restriction that the cardinality of the given subset $S\subset [n]$ is not divisible by any given integer.
Each part in Theorem 6 involves a particular restriction.

\textit{Conjecture 3:} Given any specific restriction, there exist dense vector subspaces of $\{0,1\}^{n^{2}}$ for any large enough $n$, so that the wildcard set $\gamma\times\gamma$ with $\gamma\subset [n]$ corresponding to a combinatorial line satisfies the restriction.

\bigskip
 
Conjecture 2 is a strengthening of Theorem 7 where we require to find a pair of elements, not only whose difference set is of a rectangular form, but that the difference set also yields a combinatorial line with such a rectangular wildcard set.


Conjecture 3 would likely involve extending the methods of Theorem 6 to find subspaces that force any possible given restriction on the wildcard sets. One again refers to \cite{Alo24,BZ23,EFF12}\footnote{In these papers, the arguments for $\{0,1\}^{n^2}$ are equivalently written in of the adjacency matrices of graphs.} for related questions  of dense sets containing no two elements whose symmetric differences are of certain specific types.

\section{Acknowledgements}

The author is thankful to Daniel Glasscock for discussions on related topics and to Pandelis Dodos for pointing out the reference to \cite{Alo24}. The author also acknowledges helpful discussions with John Griesmer.

\end{document}